\documentclass[12pt,reqno]{amsart}
\usepackage{graphicx, amssymb, hyperref, color,pdfsync}
\usepackage[all,cmtip]{xy}
\numberwithin{equation}{section}
\usepackage{mathrsfs}
\usepackage{tikz}
\usetikzlibrary{matrix,arrows}
\usepackage{dsfont}
\usepackage[USenglish]{babel}
\usepackage{longtable}
\usepackage{mathtools}

\input epsf
\advance\hoffset by -0.9 truecm

\begin{document}
\newcommand{\s}{\vspace{0.2cm}}

\newtheorem{teorema}{Theorem}[section]
\newtheorem{proposicao}[teorema]{Proposition}
\newtheorem{prop}[teorema]{Proposition}
\newtheorem{lema}[teorema]{Lemma}
\newtheorem{lem}[teorema]{Lemma}
\newtheorem{corolario}[teorema]{Corollary}
\newtheorem{assumption}[teorema]{Assumption}
\newtheorem{problem}[teorema]{Problem}
\newtheorem{observacao}[teorema]{Remark}
\newtheorem{rema}[teorema]{Remark}
\newtheorem{assertion}[teorema]{Assertion}
\newtheorem{result}[teorema]{Result}
\newtheorem{fact}[teorema]{Fact}
{\theoremstyle{definition}
\newtheorem{definicao}{Definition}}
\newtheorem*{theo*}{Theorem}
\newtheorem{exemplo}[teorema]{Example}
\newtheorem{rem}[teorema]{Remark}
\newtheorem{question}[teorema]{Question}
\newtheorem*{rem*}{Remark}
\newtheorem*{corollary*}{Corollary}

\title[Fermat function field lattices]{On lattices over Fermat function fields}

\author[Prando]{Rafael Froner Prando} 
\address{Instituto de Matem\'atica, Estat\'istica e Computa\c{c}\~ao Cient\'ifica, Universidade Estadual de Campinas, Campinas, SP 13083-859, Brazil.}
\email{r197034@dac.unicamp.br}

\author[Speziali]{Pietro Speziali}
\address{Instituto de Matem\'atica, Estat\'istica e Computa\c{c}\~ao Cient\'ifica, Universidade Estadual de Campinas, Campinas, SP 13083-859, Brazil.}
\email{speziali@unicamp.br}


\begin{abstract}
Function field lattices are an interesting example of algebraically constructed lattices. Their minimum distance is bounded below by a function of the gonality of the underlying function field. Known explicit examples--coming mostly from elliptic and Hermitian curves--typically meet this lower bound. In this paper, we construct, for every integer $n \geqslant 4$, a new family of lattices arising from the Fermat function field $F_n$ and the set of its $3n$ total inflection points. These lattices have rank $3n-1$, and we show that their minimum distance equals $\sqrt{2n}$, thereby exceeding the classical bound $\sqrt{2\gamma(F_n)} = \sqrt{2(n-1)}$. We also determine their kissing number, which turns out to be independent of $n$, and analyze the structure of the second shortest vectors. Our results provide the first explicit examples of function field lattices of arbitrarily large rank whose minimum distance surpasses the expected bound, offering new geometric features of potential interest for coding-theoretic and cryptographic applications.
\end{abstract}




\maketitle
\thispagestyle{empty}

\section{Introduction}

Algebraic function fields of one variable (equivalently, non-singular algebraic curves) over finite fields have a long history of applications in coding theory and cryptography~\cite{goppa,stich}. Algebraic Geometric (Goppa) codes admit a lower bound on their minimum distance--the \emph{designed minimum distance}--as a direct consequence of the Riemann-Roch theorem (see~\cite[Chapter 2]{stich}). The rich structure of function fields often allows one to determine--or at least estimate--the relevant code parameters and further properties. The relevant literature on the subject is vast; here, we cite \cite{BD, CF, KN1, M,S} as concrete examples. 

Function fields can also be used to construct lattices. This construction was introduced independently by Rosenbloom--Tsfasman--Vladut~\cite{rosen,tsfas} around 1990. Let $F|K$ be a function field of one variable, and let $\mathcal P=\{P_1,\dots,P_n\}$ be a set of $K$-rational places of $F$. Consider the group
\[
 \mathcal O_{\mathcal P}^* \coloneqq \{ z \in F^* : \operatorname{supp}(z) \subseteq \mathcal P \},
\]
and the valuation map
\[
 \varphi_{\mathcal P} : \mathcal O_{\mathcal P}^* \longrightarrow \mathbb{Z}^n, \qquad
 z \longmapsto (v_{P_1}(z), \dots, v_{P_n}(z)).
\]
The image $\Lambda_{\mathcal P} \coloneqq \operatorname{Im}(\varphi_{\mathcal P})$ is a lattice in $\mathbb{R}^n$. The degree of a non-constant function in $\mathcal O_{\mathcal P}^*$ heavily influences the norm of the corresponding lattice vector. The \emph{gonality} $\gamma$ of $F|K$ is the minimum possible degree of such a function, and it implies the lower bound
\[
 d(\Lambda_{\mathcal P}) \geqslant \sqrt{2\gamma}.
\]
Thus, function field lattices have an analogue of a ``designed minimum distance.''

Explicit constructions of function field lattices are rare. In~\cite{maharaj} and~\cite{minsha}, lattices from elliptic curves are studied, while~\cite{bottcher} analyzes the Hermitian function field. These examples attain the lower bound, and in many cases, the resulting lattices are well-rounded. Further, these lattices are generated by their vectors of minimal length (see \cite[Theorem 3.3]{maharaj} and \cite[Theorem 6.2]{bottcher}). However, function field lattices are not well-rounded in general, as proved by Ate\c{s} and Stichtenoth~\cite{ates2}.

In this note, we exhibit what, to our knowledge, is the first example of a function field lattice with arbitrarily large rank whose minimal distance \emph{exceeds} the bound $\sqrt{2\gamma}$. For $n \geqslant 4$, we consider $F\coloneqq F_n$ to be the function field of the Fermat curve
\[
 X^n + Y^n=Z^n,
\]
defined over $\mathbb{C}$, and $\mathcal P$ to be the set of its $3n$ total inflection points (called ``points at infinity'' by R\"ohrlich~\cite{david}). Indeed, the Fermat function field $F_n$ has gonality $\gamma(F_n) = n-1$ (see \cite{hart}) and the minimum distance of our lattices is equal to $\sqrt{2n}$ (see Corollary \ref{cordismin}). Also, we provide the kissing number of our lattices. Interestingly, it does not depend on $n$, a feature that distinguishes our lattices significantly from previous constructions. Our construction also works for $n =3$. However, the Fermat function field $F_3$ is elliptic, whence $\Lambda_3$ falls under the more general construction of \cite{maharaj}. In particular, $\Lambda_3$ has minimum distance equal to $2$ and it is well-rounded. 

For applications in cryptography, such as wiretap channels~\cite[Section 3.3]{sueli}, it is desirable to know which norms are achieved by lattice vectors. Then, it is natural to try and determine the second shortest vectors and whether or not they are linearly independent with the minimal length vectors~\cite[Section 3.3]{sueli}. 
The algebraic properties inherited by algebraically constructed lattices are also very important. Ideal lattices \cite[Chapter 4]{sueli} and lattices constructed via the so-called twisted homomorphism \cite{jorge} are good examples, as well as, Function field lattices, as we shall see.

This paper is structured as follows. In Section \ref{secbackground}, we provide all the necessary background definitions on lattices, function fields, and on the Fermat curve and its ``points at infinity''. 

Section \ref{secfffl} is devoted to the construction of the lattice over the Fermat function fields; here we compute its relevant parameters, such as its minimum distance, second minimum distance, kissing number and volume.

Finally, in Section \ref{secfinalrem}, we make some final remarks on our results and further possible ramifications of our investigation.

\section{Background and preliminary results}\label{secbackground}
 Our terminology and notation are standard. For more details on lattice theory, the reader is referred to \cite{sueli}; for the theory of algebraic function fields of one variable and curves, we refer to \cite{stich} and \cite{hirs}, respectively.
 \subsection{Lattices}
 \begin{definicao}
 Given a positive integer $m$ and linearly independent vectors $v_1,\dots,v_m\in\mathbb R^n$, the lattice $\Lambda$ generated by $v_1,\dots,v_m$ is defined as
 \[\Lambda\coloneqq\{a_1v_1+\cdots+a_mv_m:a_1,\dots,a_m\in\mathbb Z\}.\]
 Alternatively, a subset $\Lambda\subseteq\mathbb R^n$ is a lattice if and only if it is a discrete additive subgroup of $\mathbb R^n$. We shall focus our study on the following lattice properties.
 \begin{itemize}
 \item Given a lattice $\Lambda$, the minimum distance of $\Lambda$ is
 \[d(\Lambda)\coloneqq\min_{0\neq x\in\Lambda}||x||,\]
 where $||\cdot||$ denotes the euclidean norm. The kissing number of $\Lambda$ is
 \[K(\Lambda)\coloneqq\left|\{x\in\Lambda:||x||=d(\Lambda)\}\right|.\]

 \item Let $\Lambda\subset\mathbb R^n$ be a lattice generated by $m\leqslant n$ linearly independent vectors. The rank of $\Lambda$ is defined as $\operatorname{rank}(\Lambda)\coloneqq m$. If $m=n$, we say that $\Lambda$ is a full-rank lattice. Furthermore, $\Lambda$ is said to be well-rounded if it contains $m$ linearly independent minimal length vectors.

 \item Given a lattice $\Lambda$ generated by column vectors $v_1,\dots,v_m\in\mathbb R^n$, its generator matrix is $B=[v_1|\cdots|v_m]$. From this, the volume of $\Lambda$ is defined as $V(\Lambda)=\sqrt{\det(B^TB)}$.
 \end{itemize}
 \end{definicao}
 \begin{proposicao}\cite[Page 18]{sueli}\label{propvolquot}
 Let $\Lambda$ be a lattice of rank $m$ and $\Lambda'$ a sublattice of $\Lambda$ with the same rank. The quotient group $\Lambda/\Lambda'$ has finite order given by
 \[|\Lambda/\Lambda'|=\frac{V(\Lambda)}{V(\Lambda')}.\]
 \end{proposicao}
 \subsection{Function Fields}
	We now list some essential definitions and results which will allow us to construct function field lattices and study the relevant lattice properties. The reader can find proofs for these results in \cite[Chapter 1]{stich} and \cite[Section 3.7]{goppa}.
	\begin{definicao}
		Given a function field $F$ over a base field $K$, a place of $F|K$ is the principal and maximal ideal of a valuation ring $K\subsetneq A\subsetneq F$. The set of all places of $F$ is denoted by $\mathbb P_F$. Also, the positive integer $[A/P:K]$ is called the degree of $P$ and denoted by $\deg(P)$. If $\deg(P)=1$, we say $P$ is a $K$-rational place of $F$.
	\end{definicao}
	\begin{definicao}
		For each $P=\langle t\rangle\in\mathbb P_F$, we define a function $v_P:F^*\to\mathbb Z$ called the discrete valuation at $P$ as follows: given any $z\in F^*$, it can be written as $z=t^nu$ for some $n\in\mathbb Z$ and $u\in A^\times$, then $v_P(z)\coloneqq n$.
	\end{definicao}
	It follows from this definition that for all $a\in K^*$, $y,z\in F^*$ and $P\in\mathbb P_F$: $v_P(a)=0$ and $v_P(yz)=v_P(y)+v_P(z)$.
	\begin{definicao}
		The divisor group of a function field $F$, denoted by $\operatorname{Div}(F)$ is the free abelian group generated by the places of $F$. The degree of a divisor $D=\sum_{P\in\mathbb P_F}n_PP$ is $\deg (D)=\sum_{P\in\mathbb P_F}n_P\cdot\deg (P)$.
	\end{definicao}
	\begin{definicao}
		Let $z\in F$ and $P\in\mathbb P_F$. $P$ is a zero of $z$ if $v_P(z)>0$, and $P$ is a pole of $z$ if $v_P(z)<0$. If $Z\subseteq\mathbb P_F$ is the set of zeros of $z$ and $N\subseteq\mathbb P_F$ is the set of poles of $z$, we define the zero and pole divisors of $z$, respectively, as
		\[(z)_0\coloneqq \sum_{P\in Z}v_P(z)P\text{ and }(z)_\infty\coloneqq\sum_{P\in N}-v_P(z)P.\]
		We then say $(z)\coloneqq(z)_0-(z)_\infty$ is the principal divisor of $z$. Also, principal divisors always have degree zero (see \cite[Theorem 1.4.11]{stich}).
	\end{definicao}
 From this, the degree of a function $z\in F$ can also be defined.
	\begin{definicao}
		Given $z\in F$, its degree is given by $\deg(z)=\deg((z)_0)=\deg((z)_\infty)$, and can be computed as $\deg(z)=\frac12\sum_{P\in\mathbb P_F}|v_P(z)|$. Furthermore, the gonality $\gamma$ of a function field $F|K$ is defined as
		\[\gamma=\min\{\deg(z):z\text{ is a non-constant function in }F\}.\]
	\end{definicao}
 
 \begin{definicao}
 Let $\operatorname{Div}^0(F)$ denote the group of degree zero divisors of $F$ and $\operatorname{Princ}(F)$ the subgroup of $\operatorname{Div^0(F)}$ consisting of principal divisors. The quotient group
 \[\operatorname{Cl}^0(F)\coloneqq\operatorname{Div}^0(F)/\operatorname{Princ}(F)\]
 is well-defined and $h\coloneqq|\operatorname{Cl}^0(F)|$ is called the class number of $F$. Given $S$ a subset of $\mathbb P_F$, we denote by $\operatorname{Div}^0(S)$ and $\operatorname{Princ}(S)$ the corresponding groups whose divisors have support contained in $S$. We also write $h_S\coloneqq|\operatorname{Cl}^0(S)|$.
 \end{definicao}
	\begin{definicao}
		For a divisor $D\in\operatorname{Div}(F)$, the Riemann-Roch space associated to $D$ is
		\[\mathcal L(D)=\{z\in F:(z)\geqslant-D\}\cup\{0\},\]
		where $\geqslant$ is the partial order on $\operatorname{Div}(F)$ obtained by comparing divisors coefficient-wise. The dimension of $\mathcal L(D)$ as a $K$-vector space is denoted by $\ell(D)$.
	\end{definicao}
	\begin{definicao}
		The genus $g$ of a function field $F|K$ is defined as
		\[g\coloneqq\max\{\deg (D)-\ell(D)+1:D\in\operatorname{Div}(F)\}.\]
	\end{definicao}
	The theory of algebraic function fields of one variable is equivalent to the one of non-singular algebraic curves. More precisely:
	\begin{teorema}
		Let $F|K$ be an algebraic function field of one variable. There exists a non-singular projective curve $\mathcal C$ whose function field $K(\mathcal C)$ is ($K$-isomorphic to) $F$.
	\end{teorema}
	This theorem allows us to translate some concepts from function fields to algebraic curves. For example
	\begin{itemize}
		\item The genus of a function field is the same as the geometric genus of its associated algebraic curve.
		\item Every point $P\in \mathcal C$ corresponds to a place $P\in\mathbb P_F$. In particular, if $\mathcal C$ is a curve over the field $K$, then the points of $\mathcal C$ with entries in $K$ correspond to the $K$-rational places of $F$.
		\item If $\mathcal C$ is non-singular, a divisor of $\mathcal C$ is a formal sum of points with integer coefficients $D=\sum_{P\in \mathcal C}n_PP$ and almost all $n_P=0$.
 \item The gonality of a function field is the same as the gonality of its associated algebraic curve.
	\end{itemize}
 \begin{definicao}
 Given a non-singular plane curve $\mathcal C$ of degree $m\geqslant4$ and a divisor $D\geqslant0$ of $\mathcal C$, its specialty index $i(D)$ is the number of linearly independent adjoint curves of degree $m-3$ passing through $D$.
 \end{definicao}
 \begin{observacao}
 One can still compute the specialty index for a divisor $D\geqslant0$ of a plane singular algebraic curve using canonical adjoint curves (see \cite{hirs}).
 \end{observacao}
	\begin{teorema}[Riemann-Roch]
		Let $D\geqslant0$ be a divisor of the curve $\mathcal C$. Then
		\[\ell(D)=\deg (D)-g+1+i(D).\]
	\end{teorema}
	The last results of this section are related to methods of computing $\ell(D)$.
	\begin{definicao}
		Let $\mathcal F$, $\mathcal G$ be two plane algebraic curves and $Q$ be a non-singular point on both curves. The positive integer $I(\mathcal F\cap\mathcal G,Q)$ denotes the intersection multiplicity of $\mathcal F$ and $\mathcal G$ at $Q$. In addition, if $\mathcal F$ and $\mathcal G$ do not share a component, then $\mathcal F\cap\mathcal G=\{Q_1,\dots, Q_m\}$ and 
		\[\mathcal F\cdot\mathcal G\coloneqq\sum_{i=1}^mI(\mathcal F\cap\mathcal G,Q_i)\cdot Q_i\]
		is the intersection divisor of $\mathcal F$ and $\mathcal G$. In this case, we say that $\mathcal F$ `cuts out' the divisor $\mathcal F\cdot\mathcal G$ on $\mathcal G$.
	\end{definicao}
	\begin{teorema}[Bézout Theorem]\label{teobezout}
		Let $\mathcal F$ and $\mathcal G$ be plane algebraic curves of degree $m$ and $n$, respectively. If $\mathcal F$ and $\mathcal G$ do not share any common component, then
		\[\sum_{P\in\mathcal F\cap\mathcal G}I(\mathcal F\cap\mathcal G,P)\leqslant mn.\]
	\end{teorema}
	This theorem shows that if we can find curves $\mathcal F$ and $\mathcal G$ of degrees $m$ and $n$, respectively, such that the sum of the intersection multiplicities exceeds $mn$, they must have a common component.
	\begin{lema}\cite[Section 3.7]{goppa}\label{lemmadivint}
		Let $\mathcal F$ be a non-singular plane curve, $D\geqslant0$ a divisor of $\mathcal F$, and $\mathcal G$ a plane algebraic curve of degree $m$ such that
		\[\mathcal G\cdot\mathcal F=D+R.\]
		Then $\ell(D)$ is the dimension of the linear system of curves of degree $m$ passing through $R$, which is called the residue divisor. Furthermore, if $\mathcal H$ is a curve of degree $m$ passing through $R$, then $\frac{\mathcal H}{\mathcal G}$ is an element of $\mathcal L(D)$.
	\end{lema}

\subsection{Construction of lattices over function fields}
	A lattice over the function field $F|K$ is constructed as follows: take $n$ $K$-rational places of $F$ to form a set
	\[\mathcal P\coloneqq\{P_1,\dots,P_n\}\subseteq\mathbb P_F.\]
 From this, construct the set of functions
	\[\mathcal O_{\mathcal P}^*\coloneqq\{z\in F^*:\operatorname{supp}(z)\subseteq\mathcal P\},\]
	where $\operatorname{supp}(z)$ denotes the set of places that are either zeros or poles of $z$. It follows that $\mathcal O_\mathcal P^*$ is an abelian group with respect to multiplication. Thus, the map
	\begin{align*}
		\varphi_\mathcal P\colon(\mathcal O_\mathcal P^*,\cdot)&\to(\mathbb Z^n,+)\\
		z&\mapsto(v_{P_1}(z),\dots,v_{P_n}(z))
	\end{align*}
	is a well-defined group homomorphism. The discrete additive subgroup of $\mathbb R^n$ given by $\Lambda_\mathcal P\coloneqq\operatorname{Im}(\varphi_\mathcal P)$ is called the lattice over the function field $F|K$ generated by $\mathcal P$.
 
 In an effort to make this paper as self-contained as possible, we provide short proofs for the following basic properties. One can find more detailed arguments in \cite{ates}.
	\begin{proposicao}\label{propbasicas}
		Let $A_{n-1}=\left\{(x_1,\dots,x_n)\in\mathbb Z^n:\sum_{i=1}^nx_i=0\right\}$, then
		\begin{itemize}
			\item[(a)] $\Lambda_\mathcal P$ is a sublattice of $A_{n-1}$, thus, it is not a full-rank lattice.
			\item[(b)] $\Lambda_\mathcal P$ is an even lattice, that is, $||x||^2$ is an even integer for all $x\in\Lambda_\mathcal P$.
			\item[(c)] For all $z\in\mathcal O_\mathcal P^*$, $||z||\coloneqq||\varphi_\mathcal P(z)||\geqslant\sqrt{2\cdot\deg(z)}$, thus the minimum distance of $\Lambda_\mathcal P$ satisfies $d(\Lambda_\mathcal P)\geqslant\sqrt{2\gamma}$.
 \end{itemize}
 Additionally, if $\operatorname{Cl}^0(\mathcal P)$ is finite, then
 \begin{itemize}
 \item[(d)] $\operatorname{rank}(\Lambda_\mathcal P)=n-1$.
			\item[(e)] The volume of $\Lambda_\mathcal P$ is given by $V(\Lambda_\mathcal P)=\sqrt n\cdot h_\mathcal P$.
		\end{itemize}
	\end{proposicao}

 \begin{proof}
 \begin{itemize}
 \item[(a)] Since all principal divisors have degree $0$, it follows that
 \[0=\deg((z))=\sum_{i=1}^nv_{P_i}(z)\implies\varphi_\mathcal P(z)\in A_{n-1}.\]
 Given that $\operatorname{rank}(A_{n-1})=n-1<n$, it follows that $\operatorname{rank}(\Lambda_\mathcal P)\leqslant\operatorname{rank}(A_{n-1})<n$, therefore $\Lambda_\mathcal P$ is not a full-rank lattice.
 \item[(b)] This follows from (a) and the fact that $k^2\equiv k\pmod 2$ for any $k\in\mathbb Z$.
 \item[(c)] Write the principal divisor of $z\in\mathcal O_\mathcal P^*$ as
 \[(z)=(z)_0-(z)_\infty=(b_1Q_1+\cdots+b_sQ_s)-(c_1R_1+\cdots+c_tR_t),\]
 where $Q_i,R_j\in\mathcal P$ are distinct places and $b_i,c_j$ are positive integers for $i=1\dots,s$ and $j=1\dots,t$. Then
 \[||\varphi_\mathcal P(z)||^2=\sum_{i=1}^sb_i^2+\sum_{j=1}^tc_j^2\geqslant\sum_{i=1}^sb_i+\sum_{j=1}^tc_j=2\deg(z).\]
 Since $\gamma\leqslant\deg(z)$ for every non-constant function $z\in\mathcal O_\mathcal P^*$, it follows that $d(\Lambda_\mathcal P)\geqslant\sqrt{2\gamma}$.
 \item[(d)] Consider the divisors $P_1-P_i\in\operatorname{Div}^0(\mathcal P)$ for $i=2,\dots,n$. Since $h_\mathcal P$ is the order of $\operatorname{Cl}^0(\mathcal P)$, $h_\mathcal PP_1-h_\mathcal PP_i\in\operatorname{Princ}(\mathcal P)$ whence $h_\mathcal PP_1-h_\mathcal PP_i=(z_i)$ for some $z_i\in \mathcal O_\mathcal P^*$. The corresponding images of these functions are
 \begin{align*}
 \varphi_\mathcal P(z_2)&=(h_\mathcal P,-h_\mathcal P,0,\dots,0)\\
 \varphi_\mathcal P(z_3)&=(h_\mathcal P,0,-h_\mathcal P,\dots,0)\\
 &\hspace{.23cm}\vdots\\
 \varphi_\mathcal P(z_n)&=(h_\mathcal P,0,\dots,0,-h_\mathcal P).
 \end{align*}
 Thus, we find $n-1$ vectors of $\Lambda_\mathcal P$ which are linearly independent over $\mathbb R$. Hence, $\operatorname{rank}(\Lambda_\mathcal P)\geqslant n-1$. However, $\operatorname{rank}(\Lambda_\mathcal P)\leqslant n-1$ from (a), so our claim follows.
 \item[(e)] From Proposition \ref{propvolquot}, the volume of $\Lambda_\mathcal P$ is given by 
 \[V(\Lambda_\mathcal P)=V(A_{n-1})\cdot|A_{n-1}/\Lambda_\mathcal P|=\sqrt{n}\cdot|A_{n-1}/\Lambda_\mathcal P|.\]
 Now consider the map
 \begin{align*}
 \eta\colon A_{n-1}&\to\operatorname{Cl}^0(\mathcal P)\\
 (x_1,\dots,x_n)&\mapsto[x_1P_1+\cdots+x_nP_n],
 \end{align*}
 which is a surjective group homomorphism such that $\ker\eta=\Lambda_\mathcal P$. Therefore, $A_{n-1}/\Lambda_\mathcal P$ is isomorphic (as a group) to $\operatorname{Cl}^0(\mathcal P)$ and $h_\mathcal P=|A_{n-1}/\Lambda_\mathcal P|$, proving our claim. 
 \end{itemize}
 \end{proof}

 \begin{observacao}

 If $K$ is a finite field, then $\operatorname{Cl}^0(\mathcal P)$ is always finite \cite[Proposition 5.1.3]{stich}. If $K$ is infinite (for example, if $\operatorname{char}K=0$), then $\operatorname{Cl}^0(\mathcal P)$ may be infinite, even for a finite set of places $\mathcal P$ (see \cite[Theorem 3.1]{girard2} and \cite[Theorem 1]{girard}). One can still conduct the study of a lattice generated by a function field over an infinite base field, however (d) and (e) of Proposition \ref{propbasicas} will not hold.
 

 As principal divisors always have degree zero regardless of the base field $K$, $\Lambda_\mathcal P$ is always a sublattice of $A_{n-1}$. Thus, Proposition \ref{propvolquot} implies that $\operatorname{rank}(\Lambda_\mathcal P)<n-1$ whenever $\operatorname{Cl}^0(\mathcal P)$ is infinite.

 \end{observacao}

\section{Fermat function field lattices}\label{secfffl}
	We now construct a family of lattices over the Fermat function field. The following results are heavily reliant on Rohrlich's study \cite{david} of the group of functions generated by the ``points at infinity'' on the Fermat curves of degree $n\geqslant3$, including the fact that $\operatorname{Cl}^0(\mathcal P)$ is finite in this case. The characterization of this group as well as the functions contained in it are the base for our study of the lattice generated by these points over the Fermat function field.
 
 In the following, we always assume $n\geqslant4$, since for $n=1,2$, the Fermat function field is exactly the rational function field, which generates the $A_{n-1}$ lattice (see \cite[Subsection 2.3.1]{ates}). For $n=3$, the function field is elliptic, and the lattices generated in this case have been extensively studied in \cite{maharaj,minsha}.
 
	
	Let $\mathcal F_n$ denote the Fermat curve of degree $n\geqslant4$, which is the non-singular plane algebraic curve given by the projective equation $X^n+Y^n=Z^n$. Denote its function field by 
	\[F_n\coloneqq K(x,y)\text{ where }x^n+y^n=1\]
	with $K=\mathbb C$.
 The gonality of $F_n$ is $\gamma=n-1$ \cite[Theorem 2.1]{hart}.
	
	Following the notation in \cite{david}, $\mathcal P$ will be taken as the set of $3n$ $K$-rational points on the curve for which exactly one of the coordinates is zero, namely:
	\[a_j\coloneqq(0:\zeta^j:1),\ b_j\coloneqq(\zeta^j:0:1),\ c_j\coloneqq(\varepsilon\zeta^j:1:0)\text{ for }j=0,\dots,n-1,\]
 where $\zeta$ is a primitive $n$-th root of unity and $\varepsilon$ is a primitive $n$-th root of $-1$.

	Note that all points of type $a_j$ are contained in the line $X=0$, all points $b_j$ are on the line $Y=0$, and all of the $c_j$ are on $Z=0$. Thus, we refer to them as lying on a triangle where each of the lines corresponds to a side of this triangle. For the sake of simplicity, we set $A\coloneqq a_0+\cdots+a_{n-1},\ B\coloneqq b_0+\cdots+b_{n-1},\ C\coloneqq c_0+\cdots+c_{n-1}$.
	
	Denoting by $\Lambda_n$ the lattice from $F_n|K$ and generated by $\mathcal P$, we first investigate the minimum distance $d(\Lambda_n)$ and prove that it exceeds the expected lower bound of $\sqrt{2\gamma}=\sqrt{2(n-1)}$ for all $n\geqslant4$.
	
		

\subsection{Minimum distance}
 
	In order for a function to attain the minimum distance $\sqrt{2(n-1)}$, it must have a pole divisor of the form $p_1+\cdots+p_{n-1}$, where $p_1,\dots,p_{n-1}\in\mathcal P$ are all distinct. Therefore, we examine the Riemann-Roch space $\mathcal L(p_1+\cdots+p_{n-1})$. We have the following proposition:
	
	\begin{proposicao}\label{proprrn-1}
		Let $p_1,\dots,p_{n-1}\in\mathcal P$ be points such that $p_i\neq p_j$ if $i\neq j$. Then $\ell(p_1+\cdots+p_{n-1})=1$, except if $p_1,\dots,p_{n-1}$ are aligned.
	\end{proposicao}
	
	\begin{proof}
		We split the proof in two cases. First, let $n=4$ and consider the positive divisor $D=p_1+p_2+p_3$, with degree $\deg (D)=3$. Furthermore, the genus of $\mathcal F_4$ is $g(\mathcal F_4)=3$. Hence, by the Riemann-Roch Theorem, we have $\ell(D)=i(D)+1$. The adjoint curves of degree $n-3$ are lines in this case. Thus, if $p_1,p_2$ and $p_3$ lie on one side of the triangle, then they determine a unique line that passes through them, which implies $i(p_1+p_2+p_3)=1$ and $\ell(p_1+p_2+p_3)=2$. If $p_1,p_2$ and $p_3$ lie on more than one side of the triangle, there is no line passing through all three points. Hence $i(p_1+p_2+p_3)=0$ and $\ell(p_1+p_2+p_3)=1$.

		For $n\geqslant5$ we make use of Lemma \ref{lemmadivint}. Without loss of generality, suppose that all $n-1$ points lie on the line $X=0$. Take the line $\mathcal L=V(X)$. Then, Theorem \ref{teobezout} guarantees that $\mathcal L\cdot\mathcal F_n=A$. The residue divisor $R$ always consists of one point, meaning $\ell(p_1+\cdots+p_{n-1})$ is the dimension of the pencil of lines through a point, that is, $\ell(p_1+\cdots+p_{n-1})=2$.
			
			We may now assume that $p_1,\dots,p_{n-1}$ lie on the lines $X=0$ and $Y=0$. Consider the conic $\mathcal G=V(XY)$. It follows that $\mathcal G\cdot\mathcal F_n=A+B$ and the residue divisor $R$ is the sum of $n+1$ points with at least $2$ being on each side, since a side having only $1$ point would imply $D$ has its support contained only in that side. However, if a side has $3$ or more residue points, Theorem \ref{teobezout} implies that side must be a component of any conic passing through $R$. If both sides satisfy this condition, the conic is uniquely determined. 
			
			If one of the lines contains only $2$ points of the residue divisor, then the other contains the remaining $n-1$ points, and thus, is a component of the conic. The first $2$ points define a unique line, and as the conic must also pass through these points, it must be unique. Finally, observe that it is impossible for both sides to contain only $2$ residue points each, since the divisor $D$ would have $2n-4$ points and $2n-4>n-1$ for $n\geqslant5$. Therefore $\ell(p_1+\cdots+p_{n-1})=1$ in all cases.
			
			For $p_1,\dots,p_{n-1}$ lying on all three sides, take the cubic $\mathcal{C}=V(XYZ)$, which cuts out on $\mathcal F_n$ the divisor $\mathcal C\cdot\mathcal F_n=A+B+C$. The residue is the sum of $2n+1$ points
 with at least $3$ on each side, similarly to the previous case. 
			
			If a side contains $4$ or more points of $R$, the Bézout Theorem implies that such a side is a component of any cubic passing through the residue. If all sides satisfy this condition, then the cubic is unique. In the case that only $2$ sides have $4$ or more points of $R$, the cubic is still unique, since the $3$ points on the last side are aligned, and thus define a unique line through them.
			
			Finally, note that two different sides cannot contain only $3$ points of $R$ each. If this were the case, the last side would have to contain $2n-5$ points. Since $2n-5\geqslant n$ for $n\geqslant5$, this is impossible and $\ell(p_1+\cdots+p_{n-1})=1$ in all cases, finishing the proof.\qedhere
	\end{proof}
	
	From Proposition \ref{proprrn-1}, the only possibility for a function to have length $\sqrt{2\gamma}$ is to have a pole divisor with all points aligned. However, we claim that all functions in these Riemann-Roch spaces have either zeros or poles not in $\mathcal P$. Before proving our claim, we list some functions and divisors that will be useful:
 \medskip
	\begin{center}\def\arraystretch{1.1}
		\begin{tabular}{ll}
			Function & Divisor \\
			$x$& $A-C$ \\
			$y$& $B-C$ \\
			$x-\zeta^j$& $n\cdot b_j-C$ \\
			$y-\zeta^j$& $n\cdot a_j-C$ \\
 $\dfrac{y-\zeta^j}{x}$&$(n-1)\cdot a_j-\left(A-a_j\right)$\\
		$\dfrac{x-\zeta^j}{y}$&$(n-1)\cdot b_j-\left(B-b_j\right)$\\
 $x-\varepsilon\zeta^jy$& $(n-1)\cdot c_j-\left(C-c_j\right)$.\\
		\end{tabular}
	\end{center}
 \medskip

	We have thus obtained explicit bases for the non-trivial Riemann-Roch spaces of Proposition \ref{proprrn-1}, since
	\begin{align*}
		\left\{1,\frac{y-\zeta^j}{x}\right\}&\text{ is a base for } \mathcal L(A-a_j)\\
		\left\{1,\frac{x-\zeta^j}{y}\right\}&\text{ is a base for } \mathcal L(B-b_j)\\
		\left\{1,x-\varepsilon\zeta^jy\right\}&\text{ is a base for } \mathcal L(C-c_j).
	\end{align*}
	
	We can now prove our first major result:
	
	\begin{teorema}\label{dismin2n}
		The minimum distance of $\Lambda_n$ exceeds $\sqrt{2\gamma}=\sqrt{2(n-1)}$ for all $n\geqslant4$.
	\end{teorema}
	
	\begin{proof}
		We only need to examine the Riemann-Roch spaces with dimension greater than $1$. Given $f\in\mathcal L(A-a_j)$, its divisor has the form $(f)=(\alpha x+\beta(y-\zeta^j))-(x)$.
		with $\alpha,\beta\in K$. Note that we may assume $\alpha,\beta\neq0$, since if $\beta=0$, $f$ is constant, and if $\alpha=0$, $(f)=\left(\frac{y-\zeta^j}{x}\right)$, which has a zero of order $n-1$ and thus does not have length $\sqrt{2\gamma}$. Notice that $\alpha x+\beta(y-\zeta^j) = 0$ defines a line $\ell$ which is not the tangent at $a_j=(0:\zeta^j:1)$. It cannot coincide with $X=0$ since $\beta\neq0$; neither with $Y=0$ or $Z=0$, given that it passes through $a_j$.
		
		Therefore, $\ell$ can only possibly intersect the triangle at one point on each side. This means that for $n\geqslant5$, $\ell$ must have at least one zero outside of $\mathcal P$. In the case $n=4$, three points of intersection are impossible. Indeed, if $\ell$ intersected all three sides, its zero divisor would have the form $(\ell)_0=a_i+b_j+c_k$ for some $i,j,k\in\{0,1,2,3\}$. But this means $(1/\ell)_\infty=a_i+b_j+c_k$, and thus $1/\ell$ has three poles that are not aligned, contradicting Proposition \ref{proprrn-1}. Given that $x$ has all zeros and poles in $\mathcal P$, this implies that $f$ has at least one zero outside of $\mathcal P$, that is, $f\notin \mathcal O_\mathcal P^*$.
		
		For the remaining cases, we apply the same argument to functions $g\in\mathcal L(B-b_j)$ and $h\in\mathcal L(C-c_j)$ given by
		\[g=\alpha'+\beta'\cdot\frac{x-\zeta^j}{y}\text{ and }h=\alpha''+\beta''(x-\varepsilon\zeta^jy)\]
		with $\alpha',\beta',\alpha'',\beta''\in K^*$. 
 
		
		
		Since the only way for a function to attain length $\sqrt{2\gamma}$ is to have $\gamma$ simple zeros and $\gamma$ simple poles, this proves that the norm $\sqrt{2\gamma}$ is not achieved.
	\end{proof}

 \begin{observacao}
 In contrast with Theorem \ref{dismin2n}, the lattice $\Lambda_3$ does attain the lower bound $\sqrt{2\gamma}=2 $ for the minimum distance. Indeed, by \cite[Theorem 2]{david} the following function is an element of $\mathcal O_\mathcal P^*$
 \[g\coloneqq\sqrt[3]{(x-\zeta)(x-\varepsilon\zeta y)(x-\zeta^2)^2(x-\varepsilon\zeta^2y)^2}.\]
 Computing the principal divisor of $\frac{g}{y}$, we obtain
 \[\left(\frac{g}{y}\right)=b_1+2b_2-2c_0-c_1-(B-C)=b_2+c_2-(b_0+c_0).\]
 Thus, $\left|\left|\frac{g}{y}\right|\right|= 2 =\sqrt{2\gamma}$.
 \end{observacao}
 
	\begin{corolario}\label{cordismin}
		The minimum distance of the lattice $\Lambda_n$ is $d(\Lambda_n)=\sqrt{2n}$ for all $n\geqslant4$.
	\end{corolario}
	
	\begin{proof}
		Since $\Lambda_n$ is an even lattice, Theorem \ref{dismin2n} implies that the second shortest possible length is $\sqrt{2n}$. This is achieved, for example, by the following functions:
		\begin{center}
			\begin{tabular}{ll}
				Function & Divisor \\
				$x$& $A-C$ \\
				$y$& $B-C$ \\
				$1/x$& $C-A$ \\
				$1/y$& $C-B$ \\
				$x/y$& $A-B$ \\
				$y/x$& $B-A$ 
			\end{tabular}
		\end{center}
	\end{proof}
	
	\begin{observacao}\label{remdismin}
		Note that for a function to have norm $\sqrt{2n}$, it can either have degree $n$ with a divisor consisting of $2n$ points of order $1$ (as the examples of Corollary \ref{cordismin}), or it can have degree $n-1$ and a divisor with $2n-4$ points of order $1$ and $1$ point of order $2$, as in this case the norm would be $\sqrt{(2n-4)\cdot(\pm1)^2+(\pm2)^2}=\sqrt{2n}$.
		
		However, Proposition \ref{proprrn-1} implies that the latter case is impossible. Indeed, such a function $f$ would need to have $n-1$ simple zeros, $n-3$ simple poles and $1$ double pole or vice versa. In this case, either $f$ of $1/f$ would be an element of one of the Riemann-Roch spaces in Proposition \ref{proprrn-1}, a contradiction.
	\end{observacao}

\subsection{Kissing number and well-roundedness}

Next, we proceed to study the kissing number of $\Lambda_n$. The proof of Corollary \ref{cordismin} gives us the lower bound $K(\Lambda_n)\geqslant6$; we claim that this bound is indeed sharp for all $n\geqslant4$.

 From Remark \ref{remdismin}, to prove this, we need only to prove that no other functions whose divisor has $2n$ simple points exist in $\mathcal O_\mathcal P^*$. Once again, we prove this separately for $n\geqslant5$ and $n=4$. For $n\geqslant5$, we have an analogous version of Proposition \ref{proprrn-1}:

\begin{proposicao}\label{proprrn5}
	Let $n\geqslant5$ and $p_1,\dots,p_{n}\in\mathcal P$ be such that $p_i\neq p_j$ if $i\neq j$. Then $\ell(p_1+\cdots+p_n)=1$ unless at least $n-1$ points are aligned.
\end{proposicao}

	

\begin{proof}[Proof:]
		First assume, without loss of generality, that the points lie on $\mathcal L=V(X)$. Then $\mathcal L\cdot\mathcal F_n=A$ and $R=0$. This means there are no restrictions imposed on the lines of the linear system. That is, $\ell(p_1+\cdots+p_n)$ is the dimension of the space of all lines in the projective plane, hence $\ell(p_1+\cdots+p_n)=3$.
		
		Now, we may take the points over $X=0$ and $Y=0$. The conic $V(XY)$ cuts out the divisor $A+B$ on $\mathcal F_n$. The support of $R$ has $n$ points with at least $1$ of them lying on each side. Thus, the possible distributions for the points of the residue divisor are:
		\begin{center}
			\begin{tabular}{c|c}
				Side $1$ & Side $2$ \\
				$n-1$ & $1$ \\
				$n-2$ & $2$ \\
				$n-3$ & $3$ \\ 
				$\vdots$ & $\vdots$ \\
				$1$ & $n-1$
			\end{tabular}
		\end{center}
		In the first and last lines of the table, since $n-1\geqslant4$, one side is always a component of any conic passing through the residue. The other component may be any line that passes through the point lying on the other side, that is, $\ell(p_1+\cdots+p_n)=2$.
		
		For the second line, $n-2\geqslant3$, thus side $1$ is a component. The $2$ points remaining on side $2$ define a unique line, which means $\ell(p_1+\cdots+p_n)=1$. For all other lines of the table, side $2$ is always a component, and side $1$ always has $2$ or more points, meaning the conic is always unique and $\ell(p_1+\cdots+p_n)=1$.
		
		Finally, take the cubic $\mathcal C=V(XYZ)$ which cuts out the divisor $A+B+C$ on $\mathcal F_n$, meaning $R$ is the sum of $2n$ distinct points with at least $2$ and at most $n-1$ on each side, since $n$ such points on a side means that side has no points of $D$. The first two possible distributions of points in $\operatorname{supp}(R)$ are
		
		\begin{center}
			\begin{tabular}{c|c|c}
				Side $1$ & Side $2$ & Side $3$\\
				$2$ & $n-1$ & $n-1$\\
				$3$ & $n-1$ & $n-2$
			\end{tabular}
		\end{center}
		
		If one side has only $2$ points of $\operatorname{supp}(R)$, the only way to distribute the other $2n-2$ is to have $n-1$ on each remaining side. The fact that $n-1\geqslant4$ then implies these two sides must be components of the cubic. The last component is the line determined by the two points on the first side and the cubic is unique.
		
		Now, if one side has $3$ residue points, one of the other sides must have $n-1$ points, and the other, $n-2$. The side with $n-1$ points is a component, since $n-1\geqslant4$. Having already determined one line to be a component, we only need to find a conic through the remaining points. This means that a line now only needs to have $3$ points of intersection in order to be a component. With two determined components, the last line only needs to have $2$ points of intersection to qualify as a component of the cubic. This observation implies the uniqueness of the cubic in this case.
		
		Finally, suppose there are $k\geqslant4$ points of $R$ on the first side, implying it is a component. In this case, we must have $n\geqslant k+1$. To conclude uniqueness, we will show that at least one of the other sides is always a component. The distribution of points on the triangle has the form $k+(n-u)+(n-v)$, with $u,v\in\mathbb Z_{\geqslant0}$ such that $u+v=k$. 
 
 Consider the special cases
		\begin{align*}
			&k+\left(n-\frac{k}{2}\right)+\left(n-\frac{k}{2}\right),\text{ if $k$ is even.}\\
			&k+\left(n-\frac{k-1}{2}\right)+\left(n-\frac{k+1}{2}\right),\text{ if $k$ is odd.} 
		\end{align*}
		
		Note that if we can prove uniqueness for the case of these distributions, all of the other cases will follow, since other distributions will necessarily have more points of $\operatorname{supp}(R)$ on a single side. Hence, if that side was a component in these special cases, it will also be a component for other distributions of points.
		
		If $k$ is odd, we only need to show that $n-\frac{k-1}{2}\geqslant3$. We have
		\[n-\frac{k-1}{2}\geqslant k+1-\frac{k-1}{2}=\frac{k+3}{2}\geqslant4>3\]
		for $k\geqslant5$. And if $k$ is even, we show that $n-\frac{k}{2}\geqslant3$. We have
		\[n-\frac{k}{2}\geqslant k+1-\frac{k}{2}=\frac{k+2}{2}\geqslant3.\]
		for $k\geqslant4$. Hence, uniqueness is guaranteed and $\ell(p_1+\cdots+p_n)=1$.\qedhere
\end{proof}

\begin{teorema}\label{kisnumn5}
	The kissing number of $\Lambda_n$ is $K(\Lambda_n)=6$ for all $n\geqslant5$.
\end{teorema}

\begin{proof}
 Let $f$ be a function such that $||f|| = d(\Lambda_n)$. By Proposition \ref{proprrn5}, $f$ belongs to one of the following spaces: $\mathcal{L}(A),\mathcal L(B), \mathcal L(C)$, or to some $\mathcal L(D)$, where $D$ is of degree $n$ and whose support is comprised of $n-1$ aligned points and a point on another side of the triangle. In the former case, by symmetry, we may assume $f \in\mathcal L(C)$. Since a basis for $\mathcal L(C)$ is given by $\left\{1,x,y\right\}$,
	a function $f\in\mathcal L(C)$ has the form $f=\alpha+\beta x+\delta y$ with $\alpha,\beta,\delta\in K$. Note that at least two of these scalars must be non-zero, since if $\beta,\delta=0$, $f$ is constant, if $\alpha,\delta=0$, $(f)=(x)$ and if $\alpha,\beta=0$, $(f)=(y)$. 
 The same argument as in the proof of Theorem \ref{dismin2n} then applies, and $f$ has at least one zero outside of $\mathcal P$. For functions in $\mathcal L(A)$ or $\mathcal L(B)$, the proof is analogous. Hence, there can be no minimum length functions in $\mathcal L(A),\mathcal L(B)$ or $\mathcal L(C)$, except for the non-constant basis elements.
	
	In the latter case, let $D = p_1+\ldots +p_n$, with $p_1,\dots,p_{n-1}$ lying on one side and $p_n$ on another. Suppose we take $p_1+\dots+p_{n-1}=A-a_j$ for some $j=0,\dots,n-1$ and $p_n\notin\operatorname{supp}(A)$, then $\frac{y-\zeta^j}{x}\in\mathcal L(p_1+\dots+p_{n})$. In particular, $\left\{1,\frac{y-\zeta^j}{x}\right\}$ is a base for $\mathcal L(p_1+\dots+p_{n})$. Note that we can consider $n-1$ aligned points over $Y=0$ or $Z=0$ and construct the same bases as we did for Theorem \ref{dismin2n}. We have already shown that functions generated by those bases either are not in $\mathcal O_\mathcal P^*$ or have length $\sqrt{(n-1)^2+(n-1)}>\sqrt{2n}$.
\end{proof}

Now, let $n=4$. Consider $4$ distinct points $p_1,p_2,p_3,p_4\in\mathcal P$ and the divisor $D=p_1+p_2+p_3+p_4$, another application of the Riemann-Roch Theorem shows that $\ell(D)=3$ if all points are aligned and $\ell(D)=2$ otherwise.


The spaces of dimension $3$ have bases $\left\{1,\frac1x,\frac{y}{x}\right\},\left\{1,\frac1y,\frac{x}{y}\right\}$ and $\{1,x,y\}$, whose non-trivial linear combinations are functions not in $\mathcal O_\mathcal P^*$. For the other cases, if $3$ points are aligned and $1$ is not aligned, we have bases $\left\{1,\frac{y-\zeta^j}{x}\right\},\left\{1,\frac{x-\zeta^j}{y}\right\}$ and $\{1,x-\varepsilon\zeta^jy\}$ which also produces functions with zeros or poles not in $\mathcal P$.

Therefore, the only cases left to investigate are spaces with $2$ points of $D$ on each side, and spaces where the distribution of points of $D$ is $2,1,1$.
	
\begin{teorema}\label{teoK46}
	The kissing number of $\Lambda_4$ is $K(\Lambda_4)=6$.
\end{teorema}

\begin{proof}
	Without loss of generality, suppose the points of $\operatorname{supp}(D)$ are on $X=0$ and $Y=0$, that is $D=a_{i_1}+a_{i_2}+b_{j_1}+b_{j_2}$. We cut out $D$ with the conic $XY=0$, leaving residue $R=a_{i_3}+a_{i_4}+b_{j_3}+b_{j_4}$. Let $r_1$ be the line joining $a_{i_3}$ and $b_{j_3}$ and $r_2$ the line joining $a_{i_4}$ and $b_{j_4}$. The conic $r_1r_2=0$ passes through $R$, thus $\left\{1,\frac{r_1r_2}{xy}\right\}$ is a base for $\mathcal L(D)$.
	
	We see that $r_1$ and $r_2$ have zeros outside of $\mathcal P$, since they do not coincide with any side of the triangle and intersect the sides $X=0$ and $Y=0$ only in $R$. If they intersected $Z=0$, then $\frac{1}{r_1}$ and $\frac{1}{r_2}$ would have pole divisors with $3$ points in its support that are not aligned, which contradicts Proposition \ref{proprrn-1}.
	
	Now let $f=\alpha+\beta\frac{r_1r_2}{xy}$ with $\alpha,\beta\in K$. Once again, $\alpha,\beta\neq0$, since $f$ is constant if $\beta=0$ and if $\alpha=0$, $(f)=(r_1r_2)-(xy)$. Since $r_1r_2\notin\mathcal O_\mathcal P^*$, then $f\notin\mathcal O_\mathcal P^*$. So if $(f)=(\alpha xy+\beta r_1r_2)-(xy)$, then the conic $\alpha xy+\beta r_1r_2=0$ has zeros not in $\mathcal P$, therefore $f\notin\mathcal O_\mathcal P^*$.
	
	Suppose now $D=a_{i_1}+a_{i_2}+b_{j_1}+c_{k_1}$. If $r$ is the line joining $b_{j_1}$ and $c_{k_1}$, then the conic $xr=0$ cuts out $D$ with residue $R=a_{i_3}+a_{i_4}+q_1+q_2$ where $q_1,q_2\notin\mathcal P$. Let $s_1$ be the line joining $a_{i_3}$ and $q_1$ and $s_2$ the line joining $a_{i_4}$ and $q_2$. The conic $s_1s_2=0$ passes through $R$ and $\left\{1,\frac{s_1s_2}{xr}\right\}$ is a base for $\mathcal L(D)$. Taking $f=\alpha+\beta\frac{s_1s_2}{xr}\in\mathcal L(D)$ we can assume $\beta\neq0$, thus $(f)=(\alpha xr+\beta s_1s_2)-(xr)$. The conic $\alpha xr+\beta s_1s_2=0$ has zeros not in $\mathcal P$, and subtracting $(xr)$ from the divisor does not cancel them, meaning $f\notin\mathcal O_\mathcal P^*$.
	
	Since we have eliminated every other possibility for a function with norm $\sqrt{2n}$, it follows that $K(\Lambda_4)=6$.
\end{proof}
	
As an immediate corollary to Theorems \ref{kisnumn5} and \ref{teoK46}, we obtain the following

\begin{corolario}
	$\Lambda_n$ is not well-rounded for $n\geqslant4$.
\end{corolario}
\begin{proof}
	From \cite[Theorem 1]{david}, $\operatorname{Cl}^0(\mathcal P)$ is finite. Thus, the rank of $\Lambda_n$ for $n\geqslant4$ is $\operatorname{rank}(\Lambda_n)=3n-1\geqslant11$. Since there are only $6$ minimal vectors (only $2$ of which are linearly independent), the claim follows.
\end{proof}

\begin{observacao}
 This family of Fermat function field lattices has the property of exceeding the minimum distance lower bound and not being well-rounded, even for arbitrarily large dimensions. The lattices over elliptic function fields are not well-rounded only for low dimension (see \cite[Theorem 3.3]{maharaj}), and all the examples of function field lattices which are not well-rounded produced in \cite{ates,ates2} rely on the fact that the minimum distance is exactly $\sqrt{2\gamma}$. Also, having a fixed kissing number is a property not exhibited by elliptic or Hermitian function field lattices \cite[Theorem 3.2]{maharaj}, \cite[Theorem 6.2]{bottcher}.
\end{observacao}

\subsection{Second minimum distance}

In light of Theorem \ref{dismin2n}, it is natural to ask if the next shortest vectors of $\Lambda_n$ have length $\sqrt{2(n+1)}$ or if there are other norms which are not achieved.

We thus turn our attention to the problem of determining the second minimum distance of $\Lambda_n$. We will see that for smaller values of $n$, the vectors attaining the second minimum distance are linearly independent with the shortest vectors, whereas for larger values of $n$ this does not happen.

We denote the second minimum distance by $d_2(\Lambda_n)$. Note that some candidates are
\begin{itemize}
	\item $d_2(\Lambda_n)=\sqrt{n^2-n}$ if $n=4,5,6,7$, achieved by the functions $x-\varepsilon\zeta^jy$, for example.
	\item $d_2(\Lambda_n)=\sqrt{6n}$ for $n\geqslant7$, achieved by the function $xy$, for example.
\end{itemize}

\begin{teorema}\label{teo2distmin4}
	For $n=4$, $d_2(\Lambda_4)=\sqrt{n^2-n}=\sqrt{12}$.
\end{teorema}

\begin{proof}
	For any $f\in\mathcal O_\mathcal P^*$, $||f||\geqslant\sqrt{2\deg(f)}$. Thus, if $f$ is a function such that $d(\Lambda_4)=\sqrt8<||f||<\sqrt{12}$, then $3\leqslant\deg(f)\leqslant5$. Since $\Lambda_4$ is an even lattice, we need only to check for functions in $\mathcal O_\mathcal P^*$ which could have length $\sqrt{10}$.
	
	Suppose $\deg(f)=3$, then the only possibility is to have $p_1,p_2,p_3,p_4\in\mathcal P$ such that $(f)=2p_3+p_4-(2p_1+p_2)$, since in this case $||f||=\sqrt{2+4+4}=\sqrt{10}$. 
 From this, let $D=2p_1+p_2$. There are no lines passing through $D$, since the only line intersecting $p_1$ with order greater than $1$ is its tangent, which does not pass through $p_2$, thus $i(D)=0$ and by the Riemann-Roch Theorem, $\ell(D)=1$, proving such a function $f$ does not exist.
	
	Now, if $\deg(f)=4$, the only possibility in order to have $||f||=\sqrt{10}$ is for the principal divisor of $f$ to have the form $\pm(f)=2p_1+p_2+p_3-(p_4+p_5+p_6+p_7)$ for distinct $p_1,\dots,p_7\in \mathcal P$. In any case, either the zero or pole divisor of $f$ must consist of $4$ distinct points with order $1$, which we know to be impossible from Theorem \ref{teoK46}.
	
	Finally, if $\deg(f)=5$ and $(f)$ has $10$ simple points, length $\sqrt{10}$ is also achieved. Let us show this configuration is impossible. Take $p_1,\dots,p_5\in\mathcal P$ distinct points and $D=p_1+\dots+p_5$. Note that $\deg (D)=5\geqslant2\cdot3-1=2\cdot g(\mathcal F_4)-1$, thus $\ell(D)=\deg (D)+1-g(\mathcal F_4)=3$.
 
 
	If $D$ has $4$ aligned points, then $\mathcal L(D)=\left\langle1,\frac1x,\frac{y}{x}\right\rangle,\left\langle1,\frac1y,\frac{x}{y}\right\rangle$ or $\left\langle1,x,y\right\rangle$, which produce functions not in $\mathcal O_\mathcal P^*$.
	
	Assuming $D$ has $3$ points on a line, without loss of generality, let $D=a_{i_1}+a_{i_2}+a_{i_3}+b_{j_1}+b_{j_2}$ be cut by $XY=0$, leaving residue $R=a_{i_4}+b_{j_3}+b_{j_4}$. Taking $r_1$ the line joining $a_{i_4}$ and $b_{j_3}$ and $r_2$ the line joining $a_{i_4}$ and $b_{j_4}$, $r_1r_2=0$ passes through $R$, and $\left\{1,\frac{y-\zeta^{i_4}}{x},\frac{r_1r_2}{xy}\right\}$ is a base for $\mathcal{L}(D)$ whose non-trivial linear combinations produce functions not belonging to $\mathcal O_\mathcal P^*$ by the same arguments as before.
	
	If $D=a_{i_1}+a_{i_2}+a_{i_3}+b_{j_1}+c_{k_1}$, we cut it out by the conic $xr=0$, where $r$ is the line joining $b_{j_1}$ and $c_{k_1}$. Then $R=a_{i_4}+q_1+q_2$, with $q_1,q_2\notin\mathcal P$. Taking $s_1$ the line defined by $q_1$ and $a_{i_4}$ and $s_2$ the line defined by $q_2$ and $a_{i_4}$, the conic $s_1s_2=0$ passes through $R$, implying $\left\{1,\frac{y-\zeta^{i_4}}{x},\frac{s_1s_2}{xr}\right\}$ is a base for $\mathcal L(D)$, which produces functions with zeros or poles outside of $\mathcal P$.
	
	Finally, in the case $D=a_{i_1}+a_{i_2}+b_{j_1}+b_{j_2}+c_{k_1}$, let $s_1$ be the line joining $a_{i_1}$ and $c_{k_1}$ and $s_2$ the line joining $a_{i_2}$ and $c_{k_1}$. We cut out $D$ by the cubics $xys_1=0$ and $xys_2=0$, which produces residues $R=A-a_{i_2}+b_{j_3}+b_{j_4}+q_1+q_2$ and $R=A-a_{i_1}+b_{j_3}+b_{j_4}+q_3+q_4$, respectively. Considering $w_1$ the line joining $q_1$ and $q_2$ and $w_2$ the line joining $q_3$ and $q_4$, the cubics $xyw_1=0$ and $xyw_2=0$ pass through the respective residues. Therefore, $\left\{1,\frac{xyw_1}{xys_1},\frac{xyw_2}{xys_2}\right\}=\left\{1,\frac{w_1}{s_1},\frac{w_2}{s_2}\right\}$ is a base $\mathcal L(D)$, producing functions outside of $\mathcal O_\mathcal P^*$. Since this is the last case possible case that would achieve norm $\sqrt{10}$, it is proved that $d_2(\Lambda_4)=\sqrt{12}$.
\end{proof}

For $n\geqslant5$, we make use of the following result whose proof can be found in \cite[Corollary 1]{david}.

\begin{teorema}\label{propdr}
	A divisor of degree $0$ with support contained in $\mathcal P$ is a principal divisor if and only if the reduction of its coefficients modulo $n$ is in the span of $A, B, C$ and
	\begin{align*}
		D_1&\coloneqq\sum_{j=0}^{n-1}j(a_j+b_j)\\
		D_2&\coloneqq\sum_{j=0}^{n-1}j(b_j+c_j)\\
		D_3&\coloneqq\sum_{j=0}^{n-1}j(j+1)(a_j+b_j+c_j).
	\end{align*}
\end{teorema}
\begin{proposicao}\label{propredmodn}
	Let $f\in\mathcal O_\mathcal P^*$ be such that $||f||^2<6n$. Then, all coefficients of $(f)$ have absolute value smaller than $n$ In particular, $(f)$ is reduced modulo $n$.
\end{proposicao}
\begin{proof}
	Suppose $(f)=\sum_{j=1}^km_jp_j$ and $|m_i|\geqslant n$ for some $i\in\{1,\dots,k\}$. Then
	\[||f||^2=\sum_{j=1}^km_j^2\geqslant n^2+\sum_{\substack{j=1\\j\neq i}}^km_j^2\geqslant n^2+(k-1).\]
	If $n\geqslant6$, then $n^2\geqslant6n$ and $||f||^2>6n$, a contradiction. Now suppose that $n=5$ and some point of $(f)$ has order $5$. Since $(f)$ has degree $0$, the form that has the smallest norm is $\pm(f)=5p_1-(p_2+p_3+p_4+p_5+p_6)$, which produces $||f||^2=5^2+5=30=6n$, a contradiction. Therefore, $|m_j|\leqslant n-1$ for all $j=1,\dots,k$.
\end{proof}

As a consequence of Theorem \ref{propdr} and Proposition \ref{propredmodn}, we shall focus on the possible norms attained by divisors generated by $A,B,C,D_1,D_2,D_3$. We may write the reduction modulo $n$ of a principal divisor as
\[\overline{(f)}=m_1A+m_2B+m_3C+m_4D_1+m_5D_2+m_6D_3,\]
and all coefficients must be reduced modulo $n$, which directly implies $m_6=0$ and $m_4,m_5\in\{0,\pm1\}$. For the same reason, we cannot have $m_4=m_5=\pm1$, since the points $b_j$ appear on both $D_1$ and $D_2$. We now examine the minimum possible norm for all remaining combinations of coefficients in order to obtain the values of the second minimum distance.

In the case where $(m_4,m_5)\in\{(1,-1),(-1,1)\}$, the smallest norm occurs when $m_2=0$.  Since we want the coefficients to remain reduced modulo $n$, we must have $m_1$ and $m_4$ with opposite signs, as well as, $m_3$ and $m_5$ with opposite signs. The, by symmetry,  $|m_1|=|m_3|$. If we then suppose $m_4=1$ and $m_5=-1$, we minimize the norm of a function $f$ such that
\[\overline{(f)}=m_1A-m_1C+D_1-D_2.\]
We may write the expression for the norm as a function of $m_1$:
\[\sum_{j=0}^{n-1}(m_1+j)^2+(-m_1-j)^2=2\cdot\sum_{j=0}^{n-1}(m_1+j)^2\eqqcolon\phi(m_1).\]
To minimize $\phi$, we compute its first derivative
	\[\phi'(m_1)=\sum_{j=0}^{n-1}4(m_1+j)=\sum_{j=0}^{n-1}4m_1+\sum_{j=0}^{n-1}4j=4m_1n+2n(n-1)=2n(2m_1+n-1).\]
Thus,
\[\phi'(m_1)=0\iff2m_1+n-1=0\iff m_1=\frac{1-n}{2}.\]
If $n$ is odd, $m_1\in\mathbb Z$ and
\[\phi\left(\frac{1-n}{2}\right)=\sum_{j=0}^{n-1}2\left(\frac{1-n}{2}+j\right)^2=\frac16(n-1)n(n+1)>6n\text{ for }n\geqslant7.\]	
If $n$ is even, then $\left\lfloor\frac{1-n}{2}\right\rfloor=-\frac{n}{2}$ and $\left\lceil\frac{1-n}{2}\right\rceil=1-\frac{n}{2}$ are the integers minimizing $\phi$. In this case,
\[\phi\left(1-\frac{n}{2}\right)=\phi\left(-\frac{n}{2}\right)=\frac16n(n^2+2)>6n\text{ for }n\geqslant6.\]
Now if $m_4=-1$ and $m_5=1$, the reduction of $(f)$ is
\[\overline{(f)}=m_1A-m_1C-D_1+D_2.\]
Its norm as a function of $m_1$ is:
\[\sum_{j=0}^{n-1}(m_1-j)^2+(-m_1+j)^2=2\cdot\sum_{j=0}^{n-1}(m_1-j)^2\eqqcolon\psi(m_1),\]
with derivative $\psi'(m_1)=2n(2m_1-n+1)$. Thus,
\[\psi'(m_1)=0\iff m_1=\frac{n-1}{2}.\]
Again, if $n$ is odd, $m_1$ is an integer and
\[\psi\left(\frac{n-1}{2}\right)=\sum_{j=0}^{n-1}2\left(\frac{n-1}{2}+j\right)^2=\frac16(n-1)n(n+1)>6n\text{ for }n\geqslant7.\]	
If $n$ is even, $\left\lfloor\frac{n-1}{2}\right\rfloor=\frac{n}{2}-1$ and $\left\lceil\frac{n-1}{2}\right\rceil=\frac{n}{2}$ are the integers minimizing $\psi$. In this case,
\[\psi\left(\frac{n}{2}-1\right)=\psi\left(\frac{n}{2}\right)=\frac16n(n^2+2)>6n\text{ for }n\geqslant6.\]
Now assume that only one between $m_4$ and $m_5$ is zero. A similar argument shows that
\begin{align*}
	m_4=0&\implies m_1=0\text{ and }m_2=m_3\\
	m_5=0&\implies m_3=0\text{ and }m_1=m_2.
\end{align*}
The divisors thus obtained and the square of their norms are
\[\begin{array}{c|c|c}
	(m_4,m_5) & \text{Divisor}&\text{Norm} \\
	\hline
	(1,0)& m_1A+m_1B+D_1& 2\displaystyle\sum_{j=0}^{n-1}(m_1+j)^2\\
	\hline
	(-1,0)& m_1A+m_1B-D_1& 2\displaystyle\sum_{j=0}^{n-1}(m_1-j)^2\\
	\hline
	(0,1)& m_2B+m_2C+D_2& 2\displaystyle\sum_{j=0}^{n-1}(m_2+j)^2\\
	\hline
	(0,-1)& m_2B+m_2C-D_2& 2\displaystyle\sum_{j=0}^{n-1}(m_2-j)^2.
\end{array}\]
All of these cases have already been taken care of, leaving only the case $m_4=m_5=0$ to be examined.

Consider a divisor $D$ of the form $D=m_1A+m_2B+m_3C$. Then, $D$ has degree $0$ if $m_1+m_2+m_3=0$. However, it is possible that $m_1+m_2+m_3\neq0$. Take the divisors $m_1A,m_2B$ and $m_3C$ separately and notice that if the reduction of a divisor is equal to any of these three, then the only possibility is for the degree $0$ divisor to be one of the following principal divisors:
\begin{align*}
	\left(\prod_{j=1}^{m_1}\frac{y-\zeta^{i_j}}{x}\right)&=(n-m_1)(a_{i_1}+\cdots+a_{i_{m_1}})-m_1(a_{i_{m_1}+1}+\cdots+a_{i_n})\\
	\left(\prod_{j=1}^{m_2}\frac{x-\zeta^{i_j}}{y}\right)&=(n-m_2)(b_{i_1}+\cdots+b_{i_{m_2}})-m_2(b_{i_{m_2}+1}+\cdots+b_{i_n})\\
	\left(\prod_{j=1}^{m_3}x-\varepsilon\zeta^{i_j}y\right)&=(n-m_3)(c_{i_1}+\cdots+c_{i_{m_3}})-m_3(c_{i_{m_3}+1}+\cdots+c_{i_n}).
\end{align*}
Since there is no intersection between the supports of these divisors, the smallest norm occurs when only one of the coefficients is non-zero. Without loss of generality, we may assume only $m_1\neq0$ and write the squared norm as a function of $m_1$:
\[m_1\cdot(n-m_1)^2+(n-m_1)\cdot m_1^2=-nm_1^2+n^2m_1\eqqcolon\mu(m_1).\]
The maximum of $\mu$ occurs at $-\frac{n^2}{2\cdot(-n)}=\frac{n}{2}$, thus it is strictly increasing in the interval $\left[0,\frac{n}{2}\right]$, implying its minimum occurs at $m_1=1$. It follows that $\mu(1)=n^2-n\geqslant6n$ for $n\geqslant7$ and $\mu(1)=n^2-n<6n$ if $n=5$ or $n=6$.
	
Finally, we consider the case $m_1+m_2+m_3=0$. We initially note that the possible configurations with $m_1,m_2,m_3\in\{0,\pm1\}$ generate the $6$ functions attaining length $\sqrt{2n}$. Naturally, the smallest norms after that occur when $m_1,m_2,m_3\in\{0,\pm1,\pm2\}$. If one of the coefficients is zero, and two have absolute value $2$, the squared norm would be $n\cdot2^2+n\cdot2^2=8n$, thus the only possibility is to have one coefficient with absolute value $2$ and two with absolute value $1$, giving the norm $n\cdot1^2+n\cdot1^2+n\cdot2^2=6n$.

Since $\sqrt{6n}$ is the second shortest norm generated by these divisors, it follows that for $n\geqslant7$, the second minimum distance is indeed $\sqrt{6n}$. From our previous observations, it also follows that the second minimum distance is $\sqrt{n^2-n}$ for $n=6$. Finally, if $n=5$, the second minimum distance is the same, since the only other candidate is $\frac16(n-1)n(n+1)$, which equals $n^2-n$ when $n=5$. Thus, we have proved the following result:
\begin{teorema}\label{teo2distmin}
	$d_2(\Lambda_n)=\sqrt{n^2-n}$ for $n=4,5,6,7$ and $d_2(\Lambda_n)=\sqrt{6n}$ for $n\geqslant7$.
\end{teorema}

\begin{observacao}
 This approach can be used to give an alternative proof of Theorems \ref{dismin2n} and \ref{kisnumn5}. However, Proposition \ref{propdr} works only for the specific set of ``points at infinity'' on Fermat curves. Since the divisors considered for the minimum distance usually have only simple points, studying Riemann-Roch spaces is a line of reasoning that can be applied to any other set of points on any function field. 
\end{observacao}

\begin{observacao}
	 Theorem \ref{teo2distmin} has the following consequences. For $n=4,5,6,7$, the vectors realizing the second minimum distance are linearly independent with those realizing the minimum distance. However, for $n\geqslant8$ this does not happen, as the functions realizing the norm $\sqrt{6n}$ are $xy,x^2/y$ and $x/y^2$ and their inverses. Indeed, these vectors are linearly dependent with those attaining the minimum distance.
	
	Also, considering only divisors in the span of $A,B$ and $C$, the smallest norm after $\sqrt{6n}$ is $\sqrt{8n}$. Hence, for $n=8$ and $n=9$, the third minimum distance would be $\sqrt{n^2-n}$, which is attained by vectors that are linearly independent with those realizing the first and the second minimum distance. However, as $n$ grows, the norm of vectors linearly independent with the set that generated the minimum distance grows larger than this. In fact, they generate vectors whose squared norms scale linearly with $n$, while the first linearly independent vector has squared norm that scales quadratically with $n$.
\end{observacao}

\subsection{Volume}

As previously mentioned, the group $\operatorname{Cl}^0(\mathcal P)$ is finite for the set $\mathcal P$ of ``points at infinity''. More than that, $\operatorname{Cl}^0(\mathcal P)$ is explicitly described by the following result:

\begin{teorema}\label{teoclp}\cite[Theorem 1]{david}
If $\mathcal P$ is the set of ``points at infinity'' of a Fermat curve of degree $n\geqslant3$, then
\[\operatorname{Cl}^0(\mathcal P)\simeq\begin{cases}
 \begin{aligned}
 &\mathbb Z_n^{3n-7},&\text{ if }&n\text{ is odd}\\
 &\mathbb Z_n^{3n-7}\times\mathbb Z_2,&\text{ if }&n\text{ is even}
 \end{aligned}
\end{cases}\]
\end{teorema}
Since $|\mathcal P|=3n$, an immediate application of Proposition \ref{propbasicas}(e), yields:
\begin{teorema}\label{teovol}
 The volume of $\Lambda_n$ is
 \[V(\Lambda_n)=\begin{cases}
 \begin{aligned}
 &\sqrt{3n}\cdot n^{3n-7},&\text{ if }&n\text{ is odd}\\
 &2\sqrt{3n}\cdot n^{3n-7},&\text{ if }&n\text{ is even.}
 \end{aligned}
\end{cases}\]
\end{teorema}

\begin{observacao}
For $n =3$, Theorem \ref{teovol} agrees with the result given by \cite[Proposition 5.1]{minsha} for elliptic function fields.
\end{observacao}

\section{Final remarks}\label{secfinalrem}
We end our paper with some observations regarding our techniques, as well as, some further ramifications of our results.

First, it is worth noting that we have made our construction over $\mathbb{C}$, rather than over a finite field, as in the previous papers on the subject. Indeed, the construction of function field lattices works over any base field. In order to apply Proposition \ref{propbasicas} (d) and (e), one only needs that the group generated by the points in $\mathcal{P}$ is finite. Our Riemann-Roch space approach to compute the minimum distance works over any field $K$ such that the ``points at infinity'' are $K$-rational. It should be noted that for further results on the second minimum distance, for example, proofs for Theorems \ref{propdr} and \ref{teoclp} over fields other than $\mathbb{C}$ would be required.

Also, theoretically, the tools used in the proof of Theorem \ref{dismin2n} and Theorem \ref{kisnumn5} could be applied to compute the second minimum distance for every $n\geqslant5$. However, the number of sub-cases to be considered and discarded quickly grows as $n$ increases, making it far less practical. 

The results presented in this paper suggest that it is possible to obtain potentially nice lattice parameters by carefully choosing a relatively small subset of rational places $\mathcal P$ with some \emph{nice} geometric properties. It is natural to look for function fields with non-trivial automorphism groups, and for which $\mathcal O_\mathcal P^*$ is known for some particular choices of $\mathcal P$. It should thus be possible to obtain more examples of function field lattices whose minimum distance exceeds the expected bound. 



\section*{Acknowledgments}

The authors would like to thank Prof. Sueli Costa, for several insightful discussions on the subject of this article.

This study was financed in part by the Coordenação de Aperfeiçoamento de Pessoal de Nível Superior - Brasil (CAPES) - Finance Code 001. The first author is funded by the grant 88887.962454/2024-00.
The second author was partially supported by FAEPEX grant 2597/25.

\end{document}